\documentclass[11pt,notitlepage,twoside,a4paper]{amsart}
 \usepackage{amsfonts}
 \usepackage{mathrsfs}

\usepackage{amsmath,amssymb,enumerate}

\usepackage{epsfig,fancyhdr,color}

\usepackage{amssymb}
\usepackage{amsmath,amsthm}
\usepackage{latexsym}
\usepackage{amscd}
\usepackage{psfrag}
\usepackage{graphicx}
\usepackage[latin1]{inputenc}
\usepackage[all]{xy}
\usepackage{tikz}
\usepackage{prettyref}
\usepackage{subfigure}
\usepackage{float}

\newcommand{\T}{\mathcal{T}(S)}
\newcommand{\Tg}{\mathcal{T}(S)}
\newcommand{\M}{\mathcal{\mathrm{Mod}}(S)}

\newtheorem{theorem}{\rm\bf Theorem}[section]
\newtheorem{proposition}[theorem]{\rm\bf Proposition}
\newtheorem{lemma}[theorem]{\rm\bf Lemma}
\newtheorem{corollary}[theorem]{\rm\bf Corollary}

\theoremstyle{definition}
\newtheorem{definition}[theorem]{\rm\bf Definition}

\theoremstyle{remark}
\newtheorem{remark}[theorem]{\rm\bf Remark}

\begin{document}

\title[Thurston metric] {Thurston's metric on Teichm\"uller space and the translation lengths of mapping classes}
\author{A. Papadopoulos}
\address{Athanase Papadopoulos,  Institut de Recherche Math\'ematique Avanc\'ee, Universit{\'e} de Strasbourg and CNRS,
7 rue Ren\'e Descartes,
 67084 Strasbourg Cedex, France}
  \email{athanase.papadopoulos@math.unistra.fr}
\date{\today}

\author{W. Su}
\address{Weixu Su, Department of Mathematics, Fudan University, 200433, Shanghai, P. R. China}
\email{suwx@fudan.edu.cn}

\date{\today}

\begin{abstract} We show that the Teichm\"uller space of a surface without boundary and with punctures, equipped with Thurston's metric is the limit (in an appropriate sense) of  Teichm\"uller spaces of surfaces with boundary, equipped with their arc metrics, when the boundary lengths tend to zero. We use this to obtain a result on the translation distances for mapping classes for their actions on Teichm\"uller spaces equipped with their arc metrics.

\medskip

 \medskip

\noindent AMS Mathematics Subject Classification:   32G15 ;  30F60.

\medskip

\noindent Keywords:  Teichm\"uller space, hyperbolic geometry, quasiconformal mapping, Thurston metric, arc metric.
\end{abstract}

\maketitle

\section{Introduction}\label{sec:intro}
In this paper, we show that the arc metric on the Teichm\"uller space of surfaces with boundary limits to the Thurston metric on the Teichm\"uller space of surfaces without boundary, by making the boundary lengths tend to zero. We use this to prove a result on the translation distances for mapping classes.

We introduce some notation before stating precisely the results.

In all this paper, $S=S_{g,p,n}$ is a connected orientable surface of finite type, of genus $g$ with $p$
punctures and  $n$  boundary components.
We assume that $S$ has negative Euler characteristic, i.e., $\chi(S)=2-2g-p-n<0$.
When $n>0$, we denote by $\partial S$ the boundary of $S$.

A \emph{hyperbolic structure} on $S$ is a complete metric of constant curvature $-1$ such that
\begin{enumerate}[(i)]
\item each puncture has a neighborhood isometric to a cusp, i.e., to
the quotient $$\big\{z=x+iy \in\mathbb{H}^{2}\  | \ y>a \big\}/<z\mapsto z+1>,$$ for some $a>0$.
\item each boundary component is a simple closed geodesic.
\end{enumerate}
We denote by $\T$ the Teichm\"uller space of $S$, that is, the set of homotopy classes of hyperbolic structures on this surface.

We say that a simple closed curve on $S$  is \emph{essential} if it is neither homotopic to a point
nor to a puncture (but it can be homotopic to a boundary component).
Let $\mathcal{S}$ be the set of homotopy classes of essential simple closed curves
on $S$.

An \emph{arc} on $S$ is the homeomorphic image of a closed  interval which is  properly embedded in $S$, that is, the interior of the arc is in the interior of $S$ and the endpoints
of the arc are on the boundary of $S$. All homotopies of arcs that we consider are
relative to $\partial S$, that is, they keep the endpoints of the arc on the boundary $\partial S$ (but they do not
necessarily fix pointwise the points of $\partial S$). An arc is said to be \emph{essential}
if it is not homotopic to an arc whose image is contained in $\partial S$.
Let $\mathcal{A}$ be the set of homotopy classes of essential arcs on $S$.

Assume that $S$ is equipped with a hyperbolic structure $X$. For any $\gamma\in \mathcal{A}\cup \mathcal{S}$,
there is a unique geodesic $\gamma^X$ in its homotopy class. It is orthogonal to $\partial X$ at each endpoint, in the case where
$\gamma$ is an equivalence class of arc. We denote by
$\ell_\gamma(X)$ the length of
$\gamma^X$, and we call it the \emph{geodesic length} of $\gamma$ on $X$. This geodesic length only depends on
the equivalence class of $X$ in $\T$. Therefore it is a function defined on  $\T$.

There is an asymmetric metric, the \emph{arc metric}, on $\T$ defined by

\begin{equation} \label{eq:arc}
d(X,Y)= \sup_{\gamma\in \mathcal{A}\cup \mathcal{S}} \log \frac{\ell_\gamma(Y)}{\ell_\gamma(X)}.
\end{equation}
This metric was introduced in \cite{LPST}. It is an analogue for surfaces with boundary of the Thurston (asymmetric) metric \cite{Thurston}. The arc metric is also studied in  the papers  \cite{ALPS} \cite{LPST1}  \cite{2009j} \cite{2009h}.

By gluing a hyperbolic surface with its mirror image along the boundary components,
we obtain a natural isometric embedding of the Teichm\"uller space $\T$ into $\mathcal{T}(\widetilde{S})$, where $\widetilde{S}$ denotes the double of $S$. We endow $\mathcal{T}(\widetilde{S})$ with the Thurston metric.
The authors do not know whether such an embedding is totally geodesic (although this is unlikely).
In a recent paper \cite{ALPS}, the authors described how such an embedding extends continuously to the Thurston compactification.
Inspired by the work of Walsh \cite{Walsh}, the authors proved in \cite{ALPS} that the horofunction
boundary of $\left(\T,d\right)$ is homeomorphic to the Thurston boundary. It is conjectured in \cite{ALPS} that, without the usual exceptional cases, the isometry group of $\left(\T,d\right)$
is the extended mapping class group.

\subsection{Convergence of the arc metric}
The first aim of this paper is to study the subsets of $\T$ corresponding to hyperbolic surfaces whose boundary components have fixed lengths. These subsets are intersections of level subsets of the length functions of boundary components.

Let $\mathcal{B}=\{\beta_1,\cdots,\beta_n\}$ be the set of
 boundary components of $S$. Given any $L=(L_1, \cdots, L_n), L_i>0$, we set

$$\mathcal{T}_L(S)=\big\{X\in \T \ | \  \ell_{\beta_i}(X)=L_i, 1\leq i\leq n \big\}.$$

For a fixed vector $L=(L_1, \cdots, L_n)$,
there is a metric on $\mathcal{T}_L(S)$,
defined using the same formula  (\ref{eq:arc}).
This is  the arc metric on $\mathcal{T}_L(S)$.
It is also the metric induced by the arc metric on $\T$ on $\mathcal{T}_L(S)$ considered as a subset of $\T$.

When $L=0$, we define $\mathcal{T}_0(S)$ be the Teichm\"uller space of hyperbolic structures on $S$ such that
each boundary component is replaced by a puncture. We prove that
as $L\to 0$, the restriction of the arc metric on $\mathcal{T}_L(S)$ limits to the Thurston metric on
$\mathcal{T}_0(S)$.
This limiting behaviour is described using Fenchel-Nielsen coordinates. It is also expressed in Proposition \ref{lem:Lip} below where we prove that there is a certain map
from $\Psi_L: \mathcal{T}_L(S)$ to $\mathcal{T}_0(S) $ which is a $\big(1,o(1)\big)$-quasi-isometry.

\subsection{Application to the mapping class group}
In the second part of this paper we use the results of the first part to study the translation distances of mapping class group elements on Teichm\"uller space equipped with the Thurston metric.
This is based on Proposition \ref{lem:Lip}. Let us be more precise.

In this section, the surface is without boundary, that is, $S$ is an oriented surface obtained from a closed surface by removing finitely many points (possibly
none). We denote its mapping class group
 by $\M$.
 For an element $f\in \M$, its \emph{translation distance} with respect to the Thurston metric $d_{\mathrm{Th}}$ on $\T$ is defined by
 $$\sigma(f)= \inf_{X\in \T} d_{\mathrm{Th}}\big( X, f(X)\big).$$

The \emph{maximal dilatation} of $f$ is the largest of the dilatation constants of its pseudo-Anosov components, with respect to the Thurston decomposition of this mapping class. See \cite{FLP}.

 We show that

 \begin{theorem}\label{thm:dilatation}
Let $f\in \M$ and $\lambda(f)$ be the maximal dilatation of $f$. Then there is an integer $n$ such that
\begin{equation}\label{eq:translation}
\sigma(f^n)=\log \lambda(f^n).
\end{equation}
\end{theorem}

Note that we always have $\lambda(f^n)=n\lambda(f)$. We recall that Bers proved in \cite{Bers} that
$$\lambda(f)= \inf_{X\in \T} d_{\mathrm{Teich}}\big(X, f(X)\big),$$
where $d_{\mathrm{Teich}}$ denotes the Teichm\"uller metric.

Recall that from Thurston's classification, the mapping class $f$ is either periodic, reducible or pseudo-Anosov.
When $f$ is periodic, it has a fixed point in $\Tg$, thus $\sigma(f)=0$ and \eqref{eq:translation} is trivial.
Theorem \ref{thm:dilatation} was proved in  \cite{LPST2} when $f$ is pseudo-Anosov.
 In the present paper, we prove Theorem \ref{thm:dilatation} by showing that when $f$ is reducible, one can decrease the distance
$d_{\mathrm{Th}}\big(X, f(X)\big)$  (not necessary strictly) by taking a sequence of $X\in \Tg$ such that the lengths of the reducible curves of $f$
are shorten to zero.

\remark{In the paper \cite{LPST2}, the authors announced a proof that $\sigma(f)=\log \lambda(f)$ for any pseudo-Anosov mapping class $f$.
Unfortunately, there is a gap  in the argument. In fact, the (weaker) result which is proved there is that this holds up to taking a power of $f$. This is also a particular case of Theorem \ref{thm:dilatation} of the present paper.}

\medskip

The difficulty to understand the translation distance $\sigma(f)$ lies on the fact that the Thurston metric is not uniquely geodesic.
However,
there is another definition of translation distance, in a weak form, by setting:

$$\tau(f)=\lim_{n\to\infty}\frac{1}{n} \  d_{\mathrm{Th}}\left(X_0,f^n(X_0)\right),$$
where $X_0\in \T$ is fixed. It is not hard to check that the definition of $\tau(f)$ is 
independent on the choice of $X_0$. As a corollary of Theorem \ref{thm:dilatation}, we have

\begin{corollary}
Let $f\in \M$ and $\lambda(f)$ be the maximal dilatation of $f$. Then $\tau(f)=\log \lambda(f)$.
\end{corollary}

\noindent {\bf Acknowledgements.} The authors are partially supported by the French ANR grant FINSLER (G\'eom\'etrie de Finsler et applications, ANR-12-BS01-0009). W. Su is paritially supported by NSFC No: 11201078.

\section{Convergence of the arc metric under pinching}\label{sec:convergence}
In this section, $S=S_{g,p,n}$ is again a connected orientable surface of genus $g$ with $p$
punctures and  $n$  boundary components, and $\mathcal{B}=\{\beta_1,\cdots, \beta_n\}$ is the set of
 boundary components of $S$. Let $\mathcal{C}=\{\gamma_1,\cdots, \gamma_{3g-3+p}\}$ be a maximal collection of distinct homotopy classes of disjoint essential simple closed curves in the interior of $S$.
 The union $\mathcal{C}\cup \mathcal{B}$ is a pants decomposition of $S$. We denote the corresponding
 Fenchel-Nielsen (length-twist) coordinates on $\T$ by

 $$\big( (\ell_1,\tau_1), \cdots, (\ell_{3g-3+p},\tau_{3g-3+p})\big) \times \large( \ell_{3g-3+p+1},\cdots, \ell_{3g-3+p+n}\large).$$

Let us fix $L=\large( \ell_{3g-3+p+1},\cdots, \ell_{3g-3+p+n}\large)$. We define a map

\begin{eqnarray*}
\Psi_L: \mathcal{T}_L(S) &\to& T_0(S) \\
\big( (\ell_1,\tau_1), \cdots, (\ell_{3g-3+p},\tau_{3g-3+p})\big) \times L &\mapsto& \big( (\ell_1,\tau_1), \cdots, (\ell_{3g-3+p},\tau_{3g-3+p})\big)
\end{eqnarray*}
It is clear that $\Psi_L$ is a homeomorphism.

\begin{definition}\label{conv}
We say that a family of hyperbolic surfaces $X_L\in \mathcal{T}_L(S)$ converges to $X_0\in \mathcal{T}_0(S)$,
and we denoted this relation by $X_L\to X_0$,
if $\Psi_L(X_L)\to X$ as $L\to 0$.

\end{definition}
The definition is independent of the choice of the pants decomposition.

We shall use the following convergence criterion proved by Mondello; cf. \cite[Theorem 7.1]{Mondello} for a more general statement. In this statement, $g_L$ denotes the hyperbolic metric of $X_L\in \mathcal{T}_L(S)$.
\begin{lemma}\label{lem:criterion}
$X_L\to X_0$ if and only if there exist homeomorphisms $f_L: X_0\to X_L$ such that $(f_L)^\ast(g_L)\to g_0$ uniformly on
the compact subsets of $X_0$.
\end{lemma}

\subsection{The Thurston metric as a limit of the arc metric}
We are only interested in the situation where $L$ is sufficiently small.
We write in this case $L\ll 1$. If a quantity $A\to \infty$ uniformly as $L\to 0$, we write $A\gg 1$.
A quantity $B=B(L)$ satisfying $B\to 0$ as $L\to 0$ is denoted by $o(1)$.

Recall that the Thurston metric on $\mathcal{T}_0(S)$ is defined by
$$d_{\mathrm{Th}}(X,Y)= \sup_{\gamma\in  \mathcal{S}} \log \frac{\ell_\gamma(Y)}{\ell_\gamma(X)}.$$

Using the notation established at the beginning of this section, we prove the following:
\begin{proposition}\label{lem:Lip}
For each $L\ll 1$, the map $\Psi_L: \mathcal{T}_L(S) \to \mathcal{T}_0(S) $ is an $\big(1,o(1)\big)$ quasi-isometry.
\end{proposition}

More precisely, we have, for any fixed $L$ and for any $X$ and $Y$ in $\mathcal{T}_L(S)$,
\[d(X,Y)-o(1)\leq d_{\mathrm{Th}}(\Psi_L(X),\Psi_L(Y) )\leq d(X,Y)+ o(1) .\]

\begin{proof}
Let $X_0\in \mathcal{T}_0(S)$. We denote by $X_L=\Psi_L^{-1}(X_0)$. It is obvious that $X_L\to X_0$ in the sense of Definition \ref{conv}.

First, we estimate the lengths of simple closed curves $\ell_{\gamma}(X_L)$ in terms of $\ell_0(X_0)$.

For any boundary component $\beta$ of $S$, we denote its unique geodesic representative in $X_L$ by $\beta_L$.
By a well-known result, there is a collar neighborhood of $\beta_L$ of width $w(\beta_L)$ such that:
$$\sinh \frac{\ell_\beta(X_L)}{2} \sinh w(\beta_L)=1.$$
Note that $w(\beta_L)\to \infty$ as $L\to 0$.
It is also well known that any simple geodesic on $X_L$ that does not intersect $\beta_L$ is disjoint
from such a collar \cite{Wolpert}. In particular, any simple closed geodesic in the interior of $X_L$
is disjoint from the collar.

Similarly, we choose a standard cusp neighborhood of each puncture of $X_0$, that is, a neighborhood  isometric to
$$\{\Im z > \frac{1}{2}\}/<z\mapsto z+1>.$$ Any simple closed geodesic on $X_0$ is disjoint from the standard cusp neighborhood.

Denote the union of the standard cusp neighborhoods on $X_0$ by $C$.
Using Lemma \ref{lem:criterion},
it is not hard to see that there exist homeomorphisms $f_L: X_0\to X_L$ such that
\begin{enumerate}[(i)]
\item  $(f_L)^\ast(g_L)\to g_0$ uniformly on $X_0\setminus C$;
\item any closed geodesic in the interior of $X_L$ is contained in $X_L \setminus f_L^{-1}(C)$.
\end{enumerate}
As a result, we have

\begin{equation}\label{eq:1}
\ell_\gamma(X_L)=\big(1+ o(L)\big) \ell_\gamma(X_0), \ \forall \ \gamma\in \mathcal{S}.
\end{equation}

The next step is to estimate the length of an essential geodesic arc on $X_L$.
Consider an arbitrary $\alpha\in \mathcal{A}$.
Suppose that $\alpha$ joins two boundary components $\beta_1, \beta_2$ of $X_L$ (we may have $\beta_1=\beta_2$).
A tubular neighborhood of $\alpha \cup \beta_1 \cup \beta_2$ is a topological pair of pants,  denoted by $\mathcal{P}$.

We fist assume that $\beta_1 \neq \beta_2$. Then the boundary of $\mathcal{P}$ has three connected components, two of them being (homotopic to) $\beta_1$ and $\beta_2$.
We denote by $\gamma$ the (homotopy class of the) third one.

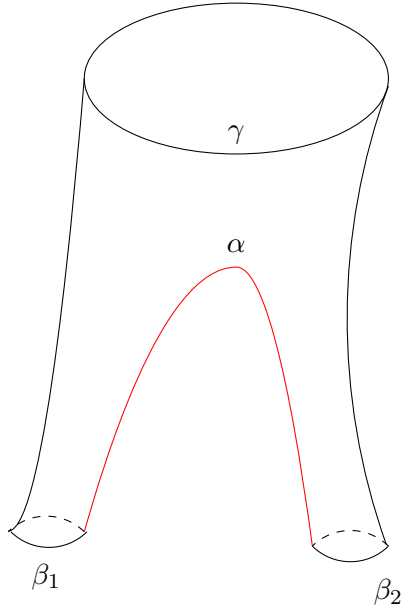
\begin{figure}[!hbp]
\centering
\begin{tikzpicture}
   \draw (0,2) ellipse (2 cm and 1 cm);
   \node[above] at (0,1) {$\gamma$};
   \draw (-3,-4) parabola (-2,2);
   \draw (-2.965,-4.01) arc (225:315: 0.707);
   \node[below] at (-2.5,-4.3) {$\beta_1$};
   \draw[dashed] (-2,-4.01) arc (45: 135: 0.707);
   \draw[red] (-2,-4.01) parabola bend (0,-0.5) (1,-4.2);
   \node[below] (0,-2) {$\alpha$};
   \draw (1,-4.2) arc (225:315: 0.707);
   \draw[dashed] (2,-4.2) arc (45: 135: 0.707);
   \draw (2,1.9) arc (160:199.5:9);
   \node[below] at (2,-4.5) {$\beta_2$};
\end{tikzpicture}
\caption{\small{The regular neighborhood is homotopic to a pair of pants.}}
\label{fig:pants}
\end{figure}

To simplify our notation, we set $\ell_\alpha=\ell_\alpha(X_L)$, etc. We use the following hyperbolic geometry formula:

\begin{equation*} \label{formula:case1}
\cosh\left(\ell_\alpha\right) = \frac{\cosh\left(\tfrac{1}{2}\ell_\gamma\right) + \cosh\left(\tfrac{1}{2}\ell_{\beta_1} \right) \cosh\left(\tfrac{1}{2}\ell_{\beta_2}\right)}{\sinh\left(\tfrac{1}{2}\ell_{\beta_1}\right)\sinh\left(\tfrac{1}{2}\ell_{\beta_2} \right)}.
\end{equation*}

By assumption, $\ell_{\beta_1}, \ell_{\beta_2} \ll 1$. By a direct calculation, we get

\begin{eqnarray}
\ell_{\alpha} &=&  \log\big( \cosh\left(\ell_\alpha\right)+ \sqrt{\cosh^2\left(\ell_\alpha\right)-1} \big)  \nonumber \\
&=& \log\big( \cosh\left(\tfrac{1}{2}\ell_\gamma\right) + \cosh\left(\tfrac{1}{2}\ell_{\beta_1} \right) \cosh\left(\tfrac{1}{2}\ell_{\beta_2}\right)\big) \nonumber \\
&& - \log\big(\sinh\left(\tfrac{1}{2}\ell_{\beta_1}\right)\sinh\left(\tfrac{1}{2}\ell_{\beta_2} \right)\big) +o(1)  \nonumber\\
&=& \ell_\gamma + |\log\big(\sinh\left(\tfrac{1}{2}\ell_{\beta_1}\right)\sinh\left(\tfrac{1}{2}\ell_{\beta_2} \right)\big)|-\log 2 + o(1) \label{eq:2}
\end{eqnarray}

The case where $\beta_1= \beta_2$ can be dealt with in the same way, by using the following formula:

\begin{align*} \label{formula:case1}
\cosh^2\left(\tfrac{1}{2}\ell_\alpha\right) &= \frac{-1 + \cosh^2\left(\tfrac{1}{2}\ell_\beta\right) + \cosh^2\left(\tfrac{1}{2}\ell_{\gamma_1}\right) + \cosh^2\left(\tfrac{1}{2}\ell_{\gamma_2}\right)}{\sinh^2\left(\tfrac{1}{2}\ell_\beta\right)} + \\
&+ \frac{2 \cosh\left(\tfrac{1}{2}\ell_\beta\right) \cosh\left(\tfrac{1}{2}\ell_{\gamma_1}\right) \cosh\left(\tfrac{1}{2}\ell_{\gamma_2}\right)}{\sinh^2\left(\tfrac{1}{2}\ell_\beta\right)}   \nonumber
\end{align*}
Here the boundary of the pair of pants $\mathcal{P}$ has three connected components, one of them being $\beta_1=\beta_2$.
We denote by $\gamma_1, \gamma_2$ the two others.

Let $Y_0\in \mathcal{T}_0(S), Y_0\neq X_0$. Denote by $Y_L=\Psi_L^{-1}(Y_0)$. Applying \eqref{eq:1},
we get
$$
\frac{\ell_\gamma(Y_L)}{\ell_\gamma(X_L)} = \big(1+ o(1)\big) \frac{\ell_\gamma(Y_0)}{\ell_\gamma(X_0)}, \ \forall \ \gamma\in \mathcal{S}.
$$
Using the fact that
$$\frac{A_1+A_2}{B_1+B_2}\leq \max\{\frac{A_1}{B_1}, \frac{A_2}{B_2}\}, \ A_1,A_2,B_1,B_2 > 0,$$
we derive from \eqref{eq:2} that  for any $\alpha\in \mathcal{A}$, there is some $\gamma\in \mathcal{S}$ such that

\begin{eqnarray*}
\frac{\ell_\alpha(Y_L)}{\ell_\alpha(X_L)}  &\leq& \max\{ \frac{\ell_\gamma(Y_L)}{\ell_\gamma(X_L)}, 1+o(1)\}\\
&\leq& \big(1+ o(1)\big) \max\{\frac{\ell_\gamma(Y_0)}{\ell_\gamma(X_0)}, 1\}.
\end{eqnarray*}
In conclusion, we have
$$\vert d(X_L,Y_L) - d_{\mathrm{Th}}(X_0,Y_0)\vert = o(1).$$
This proves the lemma.
\end{proof}

\remark{Instead of studying only the convergence of hyperbolic structures $X_L$ to a hyperbolic structure $X_0$ where all boundary components become cusps (length zero), one may study convergence to hyperbolic structures where \emph{some} of the boundary components are cusps. The same arguments we used, but with heavier notation, show that we may view the spaces $\mathcal{T}_L$ associated with  fixed vectors $L$ with some zero coordinates as sitting on the boundary of the Teichm\"uller space of the surface $S=S_{g,p,n}$ with $p$ punctures and $n$ boundary components. This boudary has an obvious stratified structure. The convergence to the boundary is in the sense of the arc metrics (each stratum is endowed with its own arc metric). This arc metric, as we prove in the present paper, becomes the Thurston metric as one approaches surfaces $S_0$ when all the boundary components become cusps.
}
\section{Translation distances of mapping classes}\label{sec:trans}

We now consider the Teichm\"uller space $\Tg$ as a space of equivalence classes of marked hyperbolic structures $(X,g)$,
where $X$ is a hyperbolic surface and $g: S\to X$ is an orientation-preserving homeomorphism (rather than a space of equivalence classes of metrics on a fixed surface). Two marked hyperbolic
structures $(X_1,g_1)$ and $(X_2,g_2)$ are equivalent if and only if there is a conformal mapping $h:X_1\to X_2$
in the homotopy class of $g_2\circ g_1^{-1}$. With this definition, a mapping class $f$ acts on $\Tg$ by changing the
markings:

$$f: (X,g) \mapsto (X,g\circ f).$$

In this section, the surface $S$ is without boundary. Its
  Teichm\"uller space is endowed  with the Thurston metric $d_{\mathrm{Th}}$.
We recall that the \emph{translation distance} of an element $f\in \mathrm{Mod}(S)$ is defined by
$$\sigma(f)=\inf_{X\in \Tg} d_{\mathrm{Th}}\big(X,f(X)\big).$$

\subsection{Reducible maps}\label{subsec:reducible} We first show that the question of understanding $\sigma(f)$
for any $f\in \mathrm{Mod}(S)$ can be reduced to that of a pseudo-Anosov mapping class $f$.

We shall pass to surfaces with boundary, those obtained by cutting $S$ along a complete reducing system for the mapping class $f$.
We shall equip the Teichm\"uller spaces of these surfaces with boundary with the arc metric and use the result of the preceding section.

Assume that $f\in \mathrm{Mod}(S)$ is reducible.
Let $\mathcal{C}=\{c_1,\cdots, c_r\}\subset \mathcal{S}$ be the maximal subset of disjoint
simple closed curves such that $f(\mathcal{C})=\mathcal{C}$. We denote the connected components of $S\setminus {\mathcal{C}}$
by $R^1, \cdots, R^s$. Each $R^j, 1\leq j\leq s$ is a surface of negative Euler characteristic with finitely many boundary
components. We choose an arbitrary hyperbolic structure on $S$ such that the length of the system
$\mathcal{C}=\{c_1,\cdots, c_r\}$ is $L=(L_1, \cdots, L_r)$ (we assume that $L_1=\cdots=L_r$, as a matter of convenience).
Note that the existence of a hyperbolic structure $X_L$ is obvious, but such a structure is not unique, since one can perform twists around the curves in $\mathcal{C}$. We let $\{X_L\}$ be the family of hyperbolic structures on
$S$ corresponding to a given $L$.

In the following, convergence $X_L\to X_0$ is in the sense of Definition \ref{conv}.

The restriction of each $X_L$ on each subsurface $R^j, 1\leq j\leq s$ induces a hyperbolic structure on
$R^j$, which we denote by $X_L^j$.

The following lemma reduces the study of $\sigma(f)$ to the consideration of a ``smaller" Teichm\"uller space.
\begin{lemma}\label{lem:limit}
As $L\to 0$, we have $X^j_L \to X^j_0, 1\leq j \leq s$  and
 $$d_{\mathrm{Th}}(X_L, f(X_L))\to \max_{j}\{ d_{\mathrm{Th}}(X^j_0, f(X^j_0))\}.$$
\end{lemma}
\begin{proof}
It is clear that $X^j_L \to X^j_0, 1\leq j \leq s$. Moreover,
the proof of Lemma \ref{lem:Lip} implies that

 $$\liminf_{L\to 0} d_{\mathrm{Th}}(X_L, f(X_L))\geq  \max_{j}\{ d_{\mathrm{Th}}(X^j_0, f(X^j_0))\}.$$

It remains to show that

 $$\limsup_{L\to 0} d_{\mathrm{Th}}(X_L, f(X_L))\leq  \max_{j}\{ d_{\mathrm{Th}}(X^j_0, f(X^j_0))\}.$$
 The proof given below is similar to our argument in the proof of Lemma \ref{lem:Lip}.

Let $\gamma\in \mathcal{S}$.  If $\gamma$ is disjoint from $\mathcal{C}$,
then $\gamma$ is contained in some subsurface $R^j$.
By Lemma \ref{lem:criterion}, $\ell_\gamma(X_L)= \left(1+o(1) \right) \ell_\gamma(X^j_0)$. This implies that

\begin{equation}\label{star}
\sup_{\gamma\in \mathcal{S}, i(\gamma,\mathcal{C})=0} \frac{\ell_\gamma(f(X_L))}{\ell_\gamma(X_L)} \leq \max_{j}\{ d_{\mathrm{Th}}(X^j_0, f(X^j_0))\}.
\end{equation}

Now we assume that $\gamma\in \mathcal{S}$ and $i(\gamma,\mathcal{C})\neq 0$.
  On the hyperbolic surface $X_L$, $\mathcal{C}$ is isotopic to a geodesic submanifold
  (the geodesic representative of $\{c_1,\cdots, c_r\}$), denoted  by
  $\widehat{\mathcal{C}}$. Up to isotopy, we can assume that $f(\mathcal{\widehat{C}})=\mathcal{\widehat{C}}$.
  The geodesic representative of $\gamma$ on $X_L$
is decomposed by $\widehat{\mathcal{C}}$ into finite many geodesic segments, denoted by $\gamma_1,\cdots, \gamma_k$.
The mapping class $f$ acts on $X_L$ by pulling back (which, restricted to $\mathcal{\widehat{C}}$, is an isometry).
If we look at the geodesic representative of $\gamma$ on $f(X_L)$, it is also decomposed by  $\mathcal{\widehat{C}}$
into geodesic segments; each of them will be denoted by $\gamma_j'$, corresponding to $\gamma_j$.
 We denote the length of each $\gamma_j$ and $\gamma_j'$ by $L_j$ and  by $L_j'$, respectively.

We assume that $\gamma_j$ is contained
in some subsurface $X^{s(j)}_L$, which is identified to a point in the Teichm\"uller space $\mathcal{T}(R_{s(j)})$ of
the subsurface $R_{s(j)}$. Since the mapping class $f$ acts on $\mathcal{T}(R_{s(j)})$, we denote the image of $X^{s(j)}_L$
 by $f(X^{s(j)}_L)$. We denote by $\ell_j=\ell(\gamma_j,X^{s(j)}_L)$, the length of the geodesic representative of the arc
$\gamma_j$ on $X^{s(j)}_L$, and $\ell_j'=\ell(\gamma_j,f(X^{s(j)}_L))$.
We claim that:

\begin{lemma}\label{lem:key} As  $L$ is sufficiently small,
\begin{equation}\label{eq:key}
\frac{L_j'}{L_j}\leq \left(1+o(1)\right)\max\left\{1, \frac{\ell_j'}{\ell_j}\right\}.
\end{equation}

\end{lemma}
The proof of Lemma \ref{lem:key} is postponed to \S \ref{sec:appendix}.
Assuming this lemma, we continue the proof of Lemma \ref{lem:limit}. Using \eqref{eq:key}, we have

\begin{eqnarray}
\frac{\ell_\gamma(f(X_L))}{\ell_\gamma(X_L)} &=&  \frac{\sum L_j'}{\sum L_j} \nonumber \\
&\leq&  (1+o(L)) \max\left\{ 1, \frac{ \ell(\gamma_j,f(X^{s(j)}_L))}{ \ell(\gamma_j,X^{s(j)}_L)} \right\}  \label{eq:oncetwist}
\end{eqnarray}

As a result, we have
\[\sup_{\gamma\in\mathcal{S},i(\gamma,\mathcal{C})\not=0}
\frac{\ell_\gamma(f(X_L))}
{\ell_\gamma(X_L)}\leq
(1+o(1))\max_{1\leq j\leq s}\{d(X_L^j,f(X_L^j))\}
.\]
By Proposition \ref{lem:Lip}, the right-hand side converges to
\[\max\{d_{\mathrm{Th}}(X_0^j,f(X_0^j)\}
.\]
The lemma follows from this result and \eqref{star}.
\end{proof}

There is a converse construction. We can first pick a point $(X_1^0, \cdots,X^s_0)$ in the product of Teichm\"uller spaces
$\mathcal{T}_0(S\setminus \mathcal{C}) := \prod_j \mathcal{T}_0(R^j)$, and construct a family of hyperbolic structures $X_L$ in $\Tg$
such that $X^j_L\to X^j_0$. As a result, we  obtain:
\begin{corollary}\label{coro:reducible}
If $f\in \mathrm{Mod}(S)$ is reduced by $\mathcal{C}=\{c_1,\cdots, c_r\}\subset \mathcal{S}$, then
$\sigma(f)$ is at most the translation distance of the action by $f$ restricted on $\mathcal{T}_0(S\setminus \mathcal{C})$.
\end{corollary}

Let $R^0$ be a connected component of $S\setminus \mathcal{C}$. There is a least integer $k$ such that $f^k(R^0)=R^0$.
We say that the union $\mathcal{R}=R_0 \cup f(R_0) \cup \cdots \cup f^{k-1}(R_0)$ is a \emph{reducible component} of $f$.
We set $$\mathcal{T}(\mathcal{R})= \prod_{j=0}^{k-1}\mathcal{T}_0(f^j(R_0))$$
and we
endow this space with the supremum Thurston metric, which we still denote by $d_{\mathrm{Th}}$.
We set
$$\sigma_{\mathcal{R}}(f)=\inf_{X\in \mathcal{T}(\mathcal{R})}d_{\mathrm{Th}}(X,f(X)).$$
By Corollary \ref{coro:reducible}, we have

$$\sigma(f)\leq \max_{\mathcal{R}} \sigma_{\mathcal{R}}(f).$$
Note that the action of $f$ on $\mathcal{R}$ is either periodic or pseudo-Anosov.
When $\mathcal{R}$ is a periodic  component of $f$, $\sigma_{\mathcal{R}}(f)=0$.

\subsection{Pseudo-Anosov maps}\label{sec:pA}
We make a remark on the translation distance of a pseudo-Anosov map $f\in \mathrm{Mod}(S)$, for the action of $f$ on Teichm\"uller space equipped with the Thurston metric,
which was investigated in \cite{LPST2}.
By definition, there is a pair of transverse measured laminations $(\mu^s, \mu^u)$ binding the surface
$S$ (which are called the \emph{stable
and unstable measured laminations} associated to $f$) satisfying:

$$f(\mu^s)=K \mu^s, f(\mu^u)=\frac{1}{K} \mu^u,$$
where $K=\lambda(f)$. By a result of Thurston (see \cite{FLP}), both $\mu^s$ and $\mu^u$ are minimal and uniquely ergodic.
Taking a sequence of simple closed curves to approximate $\mu^u$, it follows directly from the definition of the
Thurston metric that
$$\sigma(f)\geq \log \lambda(f).$$

Let us now endow $S$ with some hyperbolic structure. The complement of $S\setminus \mu^u$ consists of a finite number of ideal
polygons. We adjoin a finite number of leaves to $\mu^u$ such that it becomes a complete geodesic lamination, which we denote by
$\overline{\mu^u}$. (Note that $\overline{\mu^u}$ is not necessarily  unique, but there are finitely many choices).

Using Thurston's shearing coordinates \cite{Thurston}, there exists a unique hyperbolic structure $X\in \Tg$ corresponding
the the pair $(\overline{\mu^u},\mu^s)$ in the sense that $\mu^s$ is, up to equivalence, the horocyclic foliation of $\overline{\mu^u}$.
Then, $\{(\overline{\mu^u},e^t\mu^s)\}_{t\in \mathbb{R}}$ defines a stretch line on $\Tg$, which is a geodesic ray of the Thurston
metric passing through the point $X\cong (\overline{\mu^u},\mu^s)$.

Since the map $f$ acts isometrically between the hyperbolic surfaces $X$ and $f(X)$,
the image of the stretch line $\{(\overline{\mu^u},e^t\mu^s)\}_{t\in \mathbb{R}}$
is the stretch line $\{(f(\overline{\mu^u}), e^tf(\mu^s))\}_{t\in \mathbb{R}}$. Note that $f(\mu^s))=K\mu^s$.
There exists an integer $n$ such that $f^n(\overline{\mu^u})=\overline{\mu^u}$. It turns out that
the stretch line $\{(\overline{\mu^u},e^t\mu^s)\}_{t\in \mathbb{R}}$ is preserved by $f^n$ and the point
$X\cong (\overline{\mu^u},\mu^s)$ is transformed into $f(X)\cong (\lambda,K^n \mu^s)$.
It follows from an argument of Bers \cite{Bers} that the translation distance
$$\sigma(f^n)= \log \lambda(f^n)=n\log \lambda(f).$$

If $f$ is reducible, by Corollary \ref{coro:reducible}, the translation distance of $f$ is at most
$\max_{\mathcal{R}} \sigma_{\mathcal{R}}(f)$, where $\mathcal{R}$ is taken over all reducible component
of $f$. If $\mathcal{R}=R_0 \cup f(R_0) \cup \cdots \cup f^{k-1}(R_0)$ is a pseudo-Anosov component, then we
can take a hyperbolic structure $X_0\in \mathcal{T}(R_0)$ such that $X_0$ lies on a stretch line
preserved by some power $f^{n}$. We equip each $f^j(R_0)$ with the hyperbolic structure $f^j(X_0), 1\leq j\leq k-1$.
This shows that

$$\sigma_{\mathcal{R}}(f^n)=  n \log \lambda(f|_{\mathcal{R}}),$$
where $f|_{\mathcal{R}}$ denotes the restriction of $f$ on $\mathcal{R}$.

\begin{proof}[Proof of Theorem \ref{thm:dilatation}]
It follows directly from the definition of the
Thurston metric that
$$\sigma(f^n)\geq  n \log \lambda(f).$$
Since $f$ has finite many reducible components,
the theorem follows from the above discussion, by taking $n$ sufficiently large.

\end{proof}

\remark{When $f$ has no pseudo-Anosov component, we have $\sigma(f)=0$. The dilatation $\sigma(f)$, when $f$
is a general pseudo-Anosov map, remains an open question. Note that by Wolpert's inequality,
$d_{\mathrm{Th}}\leq 2 d_{\mathrm{Teich}}$. As a result, we have
$$\log\lambda(f)\leq \sigma(f)\leq 2\log\lambda(f).$$}

\subsection{Proof of Lemma \ref{lem:key}}\label{sec:appendix}
In this subsection, we prove Lemma \ref{lem:key}, which is used  in \S \ref{subsec:reducible}.
We also assume that $X_L$ is a family of hyperbolic structures on $S$ obtained by pinching
the curves system $\mathcal{C}$. As before, we denote the  geodesic representative of $\mathcal{C}$ on $X_L$ by
$\widehat{\mathcal{C}}$. By cutting $X_L$ along $\widehat{\mathcal{C}}$,
we obtain finitely many subsurfaces with geodesic boundary, denoted by $\{X^j_L\}$.
The mapping class $f$ acts on $X_L$, preserving the set $\widehat{\mathcal{C}}$
(with some power of Dehn twists around each component of $\widehat{\mathcal{C}}$).

\begin{lemma}\label{lem:1}
There is a uniform constant $K$ such that $f$ is homotopic to a K-quasiconformal mapping from $X_L$ to $f(X_L)$.
\end{lemma}
\begin{proof}
We use an argument of Bishop \cite{Bishop} to construct an explicit map between $X_L$ and $f(X_L)$ with the
required properties.

We choose a geodesic pants decomposition for each subsurface $X^j_L$.
For each pair of pants with boundary curves $\{\alpha_1,\alpha_2,\alpha_3\}$ (the
pair of pants is degenerate, some of the curves may be punctures), we can deform it into a new one such
that the boundary lengths $\left(\ell_{\alpha_1}(X^j_L),\ell_{\alpha_2}(X^j_L),\ell_{\alpha_3}(X^j_L)\right)$ are replaced
by $\left(\ell_{\alpha_1}(f(X^j_L)),\ell_{\alpha_2}(f(X^j_L)),\ell_{\alpha_3}(f(X^j_L))\right)$. Bishop \cite{Bishop}
proved that such a deformation can be constructed in such a way that the quasiconformal dilatation only depends on the upper bound of
$$\max\{ |\log \frac{\ell_{\alpha_1}(f(X^j_L))}{\ell_{\alpha_1}(X^j_L)}|,| \log\frac{\ell_{\alpha_2}(f(X^j_L))}{\ell_{\alpha_2}(X^j_L)}|,
|\log\frac{\ell_{\alpha_3}(f(X^j_L))}{\ell_{\alpha_3}(X^j_L)}| \}.$$
Furthermore, we can glue all the new pairs of pants together (in the same topological pattern as before) with appropriate twists such that
the resulting structure is $f(X^j_L)$. The gluing map is again a quasiconformal mapping, with dilatation controlled by
(upper bound of) the lengths and twists of the curves in the pants decomposition (see \cite{ALPSS}).

We construct the above deformation for each $X^j_L$,
and we obtain $f(X_L)$ by gluing the resulting surfaces $f(X^1_L), \cdots, f(X^s_L)$,
 in addition with some fixed power of Dehn twists
around each component of $\widehat{\mathcal{C}}$.

Note that as $L\to 0$, the hyperbolic structures $X_L$ are chosen such that on each component of $X_L\setminus \widehat{\mathcal{C}}$,
the lengths and twists of some pants decomposition  are fixed (or almost fixed). Therefore, the quasiconformal dilation of the above
deformations have a uniform upper bound independent on the choice of $L$.
\end{proof}

Consider any simple closed geodesic $\gamma$ on $X_L$ which intersects with $\widehat{\mathcal{C}}$.
Let $\gamma_j$ be a segment of $\gamma$ contained in a component, say $X_L^j$, of $X_L\setminus \widehat{\mathcal{C}}$.
We assume that the two endpoints of $\gamma_j$ lie on $\beta_1$ and $\beta_2$, respectively.
Denote the length of $\gamma_j$ by $L_j$.

It is convenient to work on the universal cover $\mathbb{H}^2$. As shown in Figure \ref{fig:int},
the geodesic $\beta_1$ is lifted to the image axes $\tilde{\beta_1}$, with endpoints $0$ and $\infty$. We can choose a universal
cover such that the point $1$ is an endpoint of $\tilde{\beta_2}$, a lift of $\beta_2$, and the geodesic segment $\gamma_j$
is realized as a geodesic segment connecting $\tilde{\beta_1}$ and $\tilde{\beta_2}$. The geodesic representative
of $\gamma_j$ on the subsurface $X^L_j$ corresponding to the geodesic segment $\hat{\gamma_j}$,
which intersects $\tilde{\beta_1}$ and $\tilde{\beta_2}$ perpendicularly.
A lift of $\gamma$, denoted by $\tilde{\gamma}$, is also drawn in the figure, with endpoints $x_1$ and $x_2$

One can see from Figure \ref{fig:int} that the geodesic segments $\gamma_j, \hat{\gamma_j}$ together with $\tilde{\beta_1}$ and $\tilde{\beta_2}$
bound a geodesic quadrilateral.

\begin{figure}[!hbp]
\centering
\begin{tikzpicture}
   \draw[->] (-6, 0) -- (6,0) node[below]{$x$};
   \draw[->] (0, 0) -- (0,5.5) node[above]{$y$};
   \draw[thick] (2.5,0) arc (0:180:0.5);
   \draw[thick,red] (2.3,0) arc (0:180:3.7);
   \draw[thick] (-4.5,0) arc (0:180:0.5);
   \draw[thick,blue] (1.9,0.5) arc (15:89:2);
   \draw (2.0,0.5) |- (1.85,0.6);
   \node[below] at (1.5,0) {$1$};
   \node[below] at (2.3,0) {$x_2$};
   \node[below] at (0,0) {$0$};
   \node[below] at (-5.2,0) {$x_1$};
   \node[left] at (0,3.3) {$y_2$};
   \node[left] at (0,2) {$y_1$};
   \draw (0,2.1) -| (0.1,1.99);
   \node[right] at (1.5,2.5) {$\gamma_j$};
   \node[left] at (0.75,1.6) {$\hat{\gamma}_j$};
   \node[above] at (-2,3.7) {$\tilde{\gamma}$};
   \draw[very thin] (0,0) -- (5.5,1.5);
   \draw[very thin] (0,0) -- (-5.5,1.5);
\end{tikzpicture}
\caption{\small{}}
\label{fig:int}
\end{figure}

The length of the sides on the left and right of the quadrilateral will be denoted by $a$ and $b$, respectively. Note that
$a=|\log{y_2} -\log{y_1}|$. We apply the hyperbolic quadrilateral formula:

$$\cosh(L_j)=-\sinh(a) \sinh(b)+\cosh(a) \cosh(b) \cosh(\ell_j).$$

As $L\to 0$, both of $L_j$ and $\ell_j$ are sufficiently large. Thus we have the following approximation:

$$L_j= a+b+\ell_j+O(1).$$

We have to compare the above data with those of $f(X_L)$.
We denote the corresponding quantities of $f(X_L)$  by
$y_1',y_2', \cdots$.
By Lemma \ref{lem:1}, there is a  $K$-quasiconformal mapping between $X_L$ and $f(X_L)$.
Such a mapping can be lifted to a quasiconformal mapping $\tilde{f}$ on $\mathbb{H}^2$,
such that it preserves the three points $0,1,\infty$. There is a uniform constant $M$ such that
for any $p,q,r,s$ on $\partial \mathbb{H}^2$,
$$\frac{1}{M} |(p,q,r,s)| \leq |(f(p),f(q),f(r),f(s))|\leq M |(p,q,r,s)|.$$
This implies that
$$\max\{1, \frac{x_1}{M}\}\leq f(x_1) \leq M x_1, \max\{\frac{|x_2|}{M},0\}\leq |f(x_2)| \leq M |x_2|.$$

Figure \ref{fig:proj} shows that $|y_2-y_2'|=O(1)$. In fact, since the collar neighborhood of $\beta_1$ is sufficiently large,
the other endpoint of $\tilde{\beta_2}$ is close to $1$ (this is also true for
$f(\tilde{\beta_2})$). Thus their projections on $\tilde{\beta_1}$ is near $i$, up to an uniformly bounded distance.

On the other hand, we show that $|y_1-y_1'|=O(1)$. This can be seen by a direct calculation.
Note that the two endpoints of $\tilde{\gamma}$ are $x_1$ and $x_2$. It is easy to see that
$y_1=\sqrt{|x_1x_2|}$. Under deformation by $\tilde{f}$, we have $y_1'=\sqrt{|x_1'x_2'|}$.
\begin{figure}[!hbp]
\centering
\begin{tikzpicture}
   \draw[->] (-6, 0) -- (6,0) node[below]{$x$};
   \draw[->] (0, 0) -- (0,5.5) node[above]{$y$};
   \draw[thick] (2.5,0) arc (0:180:0.5);
   \draw[thick,blue] (1.5,0) arc (0:90:1.5);
   \draw[thick,blue] (2.5,0) arc (0:90:2.5);
   \draw[very thin] (0,0) -- (5.5,1.5);
   \draw[very thin] (0,0) -- (-5.5,1.5);
   \node[below] at (1.5,0) {$1$};
   \draw[thick,red] (0,1.5) -- (0,2.5);
   \draw[->] (-0.7,2)-- (-0.2,2.2);
   \node[left] at (-0.7,2) {$O(1)$};
\end{tikzpicture}
\caption{\small{}}
\label{fig:proj}
\end{figure} As we noticed before,
$$\frac{1}{M}|x_1|\leq |x_1'|\leq M |x_1|, \frac{1}{M}|x_2|\leq |x_2'|\leq M |x_2|.$$
This means that $|\log \frac{y_1'}{y_1}|\leq M$.
As a corollary, under the deformation $\tilde{f}$, $a$ is changed up to a bounded amount, that is, $|a'-a|=O(1)$.
The same argument applies to $|b'-b|$.

Applying the quadrilateral formula again:
$$\cosh(L'_j)=-\sinh(a'_1) \sinh(a'_2)+\cosh(a'_1) \cosh(a'_2) \cosh(\ell'_j).$$
The above computations show that:

$$L_j= a+b+\ell_j+O(1),$$
$$L'_j= a'+b'++\ell'_j+O(1),$$
$$|a-a'|=O(1),|b-b'|=O(1).$$
Thus

$$\frac{L_j'}{L_j}\leq (1+o(1))\max\{1, \frac{\ell_j'}{\ell_j}\}.$$
This complete the proof of Lemma \ref{lem:key}.

\end{document}